\documentclass[a4paper,11pt]{amsart}
\usepackage{hyperref,fancyhdr,mathrsfs,amsmath,amscd,amsthm,amsfonts,latexsym,amssymb,stmaryrd}
\usepackage[all]{xy}

\DeclareMathOperator{\rank}{rank}

\DeclareMathOperator{\dif}{d}

\newcommand{\Cal}{\mathcal{C}}
\renewcommand{\H}{\mathscr{H}}

\newcommand{\F}{\mathscr{F}}

\newcommand{\tG}{\mathscr{G}}

\newcommand{\ol}{\mathcal{O}}

\def \a{\alpha}

\def \e{\eta}
\def \ep{\varepsilon}

\def \o{\omega}

\def \phi{\varphi}
\def \Phi{\varPhi}
\def \p{\pi}
\def \r{\rho}
\def \s{\sigma}
\def \t{\tau}

\def \R{\mathbb{R}}
\def \Hq{\mathbb{H}\,}
\def \C{\mathbb{C}\,}

\def\widecheckg{g^{\hspace*{-2.5pt}\vbox to 5pt{\hbox to
0pt{\LARGE$\check{}$}}}\hspace*{2pt}}

\def\widecheckl{\lambda^{\hspace*{-3.5pt}\vbox to 8pt{\hbox to
0pt{\LARGE$\check{}$}}}\hspace*{2pt}}

\begin{document}

\title{Twistor Theory for co-CR quaternionic manifolds and related structures}
\author{Stefano~Marchiafava and Radu~Pantilie}
\thanks{R.P.\ acknowledges that this work was partially supported by the Visiting Professors Programme of GNSAGA-INDAM of C.N.R. (Italy).\\
\indent
S.M.\ acknowledges that this work was done under the program of GNSAGA-INDAM of C.N.R. and PRIN07 ''Geometria Riemanniana e strutture
differenziabili'' of MIUR (Italy).}
\email{\href{mailto:marchiaf@mat.uniroma1.it}{marchiaf@mat.uniroma1.it},
       \href{mailto:radu.pantilie@imar.ro}{radu.pantilie@imar.ro}}
\address{S.~Marchiafava, Dipartimento di Matematica, Istituto ``Guido~Castelnuovo'',
Universit\`a degli Studi di Roma ``La Sapienza'', Piazzale Aldo~Moro, 2 - I 00185 Roma - Italia}
\address{R.~Pantilie, Institutul de Matematic\u a ``Simion~Stoilow'' al Academiei Rom\^ane,
C.P. 1-764, 014700, Bucure\c sti, Rom\^ania}

\newtheorem{thm}{Theorem}[section]
\newtheorem{lem}[thm]{Lemma}
\newtheorem{cor}[thm]{Corollary}
\newtheorem{prop}[thm]{Proposition}

\theoremstyle{definition}

\newtheorem{defn}[thm]{Definition}
\newtheorem{rem}[thm]{Remark}
\newtheorem{exm}[thm]{Example}

\numberwithin{equation}{section}

\maketitle
\thispagestyle{empty}
\vspace{-10mm}
\section*{Abstract}
\begin{quote}
{\footnotesize
In a general and non metrical framework, we introduce the class of co-CR quaternionic manifolds,
which contains the class of quaternionic manifolds,
whilst in dimension three it particularizes to give the Einstein--Weyl spaces.
We show that these manifolds have a rich natural Twistor Theory and, along the way, we obtain a heaven space construction
for quaternionic-K\"ahler ma\-ni\-folds.}
\end{quote}

\section*{Introduction}

\indent
Over any three-dimensional conformal manifold $M$, endowed with a conformal connection, there is a sphere bundle $Z$
endowed with a natural CR structure \cite{LeB-CR_twistors}\,. Furthermore, if $M$ is real analytic then \cite{LeB-Hspace}
the CR structure of $Z$ is induced by a germ unique embedding of $Z$ into a three-dimensional complex manifold $\widetilde{Z}$ which
is the twistor space of an anti-self-dual manifold $\widetilde{M}$\,; accordingly, $M$ is a hypersurface in $\widetilde{M}$,
and the latter is called the \emph{heaven space} (due to \cite{New-heaven}\,; cf.\ \cite{LeB-CR_twistors}\,)
of $M$ (endowed with the given conformal connection).\\
\indent
In \cite{fq_1} (see Section \ref{section:cr_q}\,), we obtained the higher dimensional versions of these constructions by introducing the notion
of \emph{CR quaternionic manifold}.
Thus, the generic submanifolds of codimensions at most $2k-1$\,, of a quaternionic manifold of dimension $4k$\,, are endowed with
natural CR quaternionic structures. Moreover, assuming real-analyticity, any CR~quaternionic manifold is obtained this way through a germ unique
embedding into a quaternionic manifold \cite{fq_1}\,.\\
\indent
Returning to the three-dimensional case, by \cite{Hit-complexmfds}\,, if the inclusion of $M$ into $\widetilde{M}$ admits
a retraction which is twistorial (that is, its fibres correspond to a (one-dimensional) holomorphic foliation on $\widetilde{Z}$\,)
then the connection used to construct the CR structure on $Z$ may be assumed to be a Weyl connection;
moreover, there is a natural correspondence between such retractions and Einstein-Weyl connections on $M$.
Furthermore, any Einstein--Weyl connection $\nabla$ on $M$ determines a complex surface $Z_{\nabla}$ and a
holomorphic submersion from $\widetilde{Z}$ onto it; then $Z_{\nabla}$ is the twistor space of $(M,\nabla)$ \cite{Hit-complexmfds}\,.\\
\indent
Furthermore, the correspondence between Einstein--Weyl spaces and their twistor spaces is similar
to the correspondence between anti-self-dual manifolds and their twistor spaces (see, also, \cite{LouPan-II}\,).
Also, from the point of view of Twistor Theory, the
anti-self-dual manifolds are just four-dimensional quaternionic manifolds (see \cite{IMOP}\,).\\
\indent
This raises the obvious question: \emph{is there a natural class of manifolds, endowed with twistorial structures, which
contains both the quaternionic manifolds and the three-dimensional Einstein--Weyl spaces}\,?\\
\indent
In this paper, where the adopted point of view is essentially non metrical, we answer in the affirmative to this question
by introducing, in a general framework, the notion
of \emph{co-CR~quaternionic manifolds} and we initiate the study of their twistorial properties.
This notion is based on the \emph{(co-)CR quaternionic vector spaces} which were introduced and classified in \cite{fq_1}
(see Section \ref{section:review_lin_co-CR_q}\,, and, also, Appendix \ref{appendix:(co-)cr_q_intrinsic} for an alternative definition)
and, up to the integrability, it is dual to the notion of
\emph{CR quaternionic manifolds}.\\
\indent
An interesting situation to consider is when a manifold may be endowed with both a CR quaternionic and a co-CR quaternionic structure
which are \emph{compatible}.
This gives the notion of \emph{$f$-quaternionic} manifold, which has two twistor spaces. The simplest example
is provided by the three-dimensional Einstein--Weyl spaces, endowed with the
twistorial structures of \cite{LeB-CR_twistors} and \cite{Hit-complexmfds}\,, respectively; furthermore,
the above mentioned twistorial retraction admits a natural generalization to the $f$-quaternionic manifolds (Corollary \ref{cor:f-q}\,).
Also, the quaternionic manifolds may be characterised as $f$-quaternionic manifolds for which the two twistor spaces coincide.\\
\indent
Other examples of $f$-quaternionic manifolds are the Grassmannnian ${\rm Gr}_3^+(l+3,\R)$ of oriented three-dimensional
vector subspaces of $\R^{l+3}$ and the flag manifold ${\rm Gr}_2^0(2n+2,\C\!)$ of two-dimensional complex vector subspaces of $\C^{\!2n+2}(=\Hq^{\!n+1})$
which are isotropic with respect to the underlying complex symplectic structure of $\C^{\!2n+2}$, $(l,n\geq1)$\,.
The twistor spaces of their underlying co-CR quaternionic structures are the hyperquadric $Q_{l+1}$ of isotropic one-dimensional
complex vector subspaces of $\C^{\!l+3}$ and ${\rm Gr}_2^0(2n+2,\C\!)$ itself, respectively.
Also, their heaven spaces are the Wolf spaces
${\rm Gr}_4^+(l+4,\R)$ and ${\rm Gr}_2(2n+2,\C\!)$\,, respectively (see Examples \ref{exm:Gr_3^+} and \ref{exm:Gr_2^0symplectic}\,,
for details). Another natural class of $f$-quaternionic manifolds is described in Example \ref{exm:SO(Q)}\,.\\
\indent
The notion of almost $f$-quaternionic manifold appears, also, in a different form, in \cite{KasMCabTri}\,.
However, there it is not considered any adequate integrability condition. Also, in \cite{Bi-qK_heaven}\,, \cite{AleKam-Annali08} and \cite{BejFar}
are considered, under particular dimensional assumptions and/or in a metrical framework, particular classes of almost $f$-quaternionic manifolds.\\
\indent
Let $N$ be the heaven space of a real analytic $f$-quaternionic manifold $M$, with $\dim N=\dim M+1$\,.
If the connection of the $f$-quaternionic structure on $M$ is induced by a torsion free connection on $M$ then the
twistor space of $N$ is endowed with a natural holomorphic distribution of codimension one which is transversal to the
twistor lines corresponding to the points of $N\setminus M$. Furthermore, this construction also works if, more generally, $M$
is a real analytic CR~quaternionic manifold which is a \emph{q-umbilical} hypersurface of its heaven space $N$.
Then, under a non-degeneracy condition, this distribution defines a holomorphic contact structure on the twistor
space of $N$. Therefore, according to \cite{LeB-twist_qK}\,, it determines a quaternionic-K\"ahler structure
on $N\setminus M$ (cf.\ \cite{Bi-qK_heaven}\,,\cite{Du-qK_dim7}\,).\\
\indent
It is well known (see, for example, \cite{PanWoo-sd} and the references therein) that the three-dimensional Einstein--Weyl spaces
are one of the basic ingredients in constructions of anti-self-dual (Einstein) manifolds. One of the aims of
this paper is to give a first indication that the study of co-CR quaternionic manifolds will lead to a better understanding
of quaternionic(-K\"ahler) manifolds.

\section{Brief review of (co-)CR quaternionic vector spaces} \label{section:review_lin_co-CR_q}

\indent
The group of automorphisms of the (unital) associative algebra of quaternions $\Hq$ is ${\rm SO}(3)$ acting trivially on $\R\,(\subseteq\Hq)$
and canonically on ${\rm Im}\Hq$.\\
\indent
A \emph{linear hypercomplex structure} on a (real) vector space $E$ is a morphism of associative algebras $\s:\Hq\to{\rm End}(E)$\,.
A \emph{linear quaternionic structure} on $E$ is an equivalence class of linear hypercomplex structures, where two
linear hypercomplex structures $\s_1,\s_2:\Hq\to{\rm End}(E)$ are \emph{equivalent} if there exists $a\in{\rm SO}(3)$ such that
$\s_2=\s_1\circ a$. A \emph{hypercomplex/quaternionic vector space} is a vector space endowed with a
linear hypercomplex/quaternionic structure (see \cite{AleMar-Annali96}\,,\,\cite{IMOP}\,).\\
\indent
If $\s:\Hq\to{\rm End}(E)$ is a linear hypercomplex structure on a vector space $E$ then the unit sphere $Z$ in
$\s({\rm Im}\Hq)\subseteq{\rm End}(E)$ is the corresponding space of \emph{admissible linear complex structures}.
Obviously, $Z$ depends only of the linear quaternionic structure determined by $\s$.\\
\indent
Let $E$ and $E'$ be quaternionic vector spaces and let $Z$ and $Z'$ be the corresponding spaces of admissible linear complex structures.
A linear map $t:E\to E'$ is \emph{quaternionic}, with respect to some function $T:Z\to Z'$, if $t\circ J=T(J)\circ t$, for any $J\in Z$
(see \cite{AleMar-Annali96}\,). If, further, $t\neq0$ then $T$ is unique and an orientation preserving isometry (see \cite{IMOP}\,).\\
\indent
The basic example of a quaternionic vector space is $\Hq^{\!k}$ endowed with the linear quaternionic structure given by its canonical
(left) $\Hq$-module structure. Moreover, for any quaternionic vector space of dimension $4k$ there exists a quaternionic
linear isomorphism from it onto $\Hq^{\!k}$. The group of quaternionic linear automorphisms of $\Hq^{\!k}$ is ${\rm Sp}(1)\cdot{\rm GL}(k,\Hq)$
acting on it by $\bigl(\pm(a,A),x\bigr)\mapsto axA^{-1}$\,, for any $\pm(a,A)\in{\rm Sp}(1)\cdot{\rm GL}(k,\Hq)$ and $x\in\Hq^{\!k}$.
If we restrict this action to ${\rm GL}(k,\Hq)$ then we obtain the group of hypercomplex linear automorphisms of $\Hq^{\!k}$.\\
\indent
If $\s:\Hq\to{\rm End}(E)$ is a linear hypercomplex structure then $\s^*:\Hq\to{\rm End}(E^*)$\,, where $\s^*(q)$ is the transpose
of $\s(\overline{q})$\,, $(q\in\Hq)$\,, is \emph{the dual linear hypercomplex structure}. Accordingly, we define the
dual of a linear quaternionic structure.

\begin{defn}[\,\cite{fq_1}\,]
A \emph{linear co-CR quaternionic structure} on a vector space $U$ is a pair $(E,\r)$\,, where $E$ is a quaternionic vector space
and $\r:E\to U$ is a surjective linear map such that ${\rm ker}\,\r\cap J({\rm ker}\,\r)=\{0\}$\,, for any admissible linear complex
structure $J$ on~$E$\,.\\
\indent
A \emph{co-CR quaternionic vector space} is a vector space endowed with a linear co-CR quaternionic structure.
\end{defn}

\indent
Dually, a \emph{CR quaternionic vector space} is a triple $(U,E,\iota)$\,, where $E$ is a quaternionic vector space
and $\iota:U\to E$ is an injective linear map such that ${\rm im}\,\r+ J({\rm im}\,\r)=E$\,, for any admissible linear complex
structure $J$ on $E$\,.\\
\indent
A map $t:(U,E,\r)\to(U',E',\r')$ between co-CR quaternionic vector spaces is \emph{co-CR quaternionic linear}
(with respect to some map $T:Z\to Z'$\,) if there exists a map $\widetilde{t}:E\to E'$ which is quaternionic linear
(with respect to $T$) such that $t\circ\r=\r'\circ\widetilde{t}$.\\
\indent
By duality, we also have the notion of \emph{CR quaternionic linear map}.\\
\indent
Note that, if $(U,E,\iota)$ is a CR quaternionic vector space then the inclusion $\iota:U\to E$ is CR quaternionic linear.
Dually, if $(U,E,\r)$ is a co-CR quaternionic vector space then the projection $\r:E\to U$ is co-CR quaternionic linear.\\
\indent
By working with pairs $(U,E)$\,, where $E$ is a quaternionic vector space and $U\subseteq E$ is a real vector subspace,
we call $({\rm Ann}\,U,E^*)$ \emph{the dual pair} of $(U,E)$\,, where the annihilator ${\rm Ann}\,U$ is formed of those $\a\in E^*$
such that $\a|_U=0$\,.\\
\indent
Any CR quaternionic vector space $(U,E,\iota)$ corresponds to the pair $({\rm im}\,\iota,E)$\,, whilst any
co-CR quaternionic vector space $(U,E,\r)$ corresponds to the pair $({\rm ker}\,\r,E)$\,. These associations define
functors in the obvious way.\\
\indent
To any pair $(U,E)$ we associate a (coherent analytic) sheaf over $Z$ as follows. Let $E^{0,1}$ be the holomorphic
vector bundle over $Z$ whose fibre over any $J\in Z$ is the $-{\rm i}$ eigenspace of $J$. Let $u:E^{0,1}\to Z\times(E/U)^{\C}$ be the
composition of the inclusion $E^{0,1}\to Z\times E^{\C}$ followed by the projection $Z\times E^{\C}\to Z\times(E/U)^{\C}$.

\begin{defn}[\,\cite{vq}\,]
$\mathcal{U}=\mathcal{U}_-\oplus\mathcal{U}_+$ is \emph{the sheaf of $(U,E)$}\,,
where $\mathcal{U}_-={\rm ker}\,u$ and $\mathcal{U}_+={\rm coker}\,u$\,.
\end{defn}

\indent
If $(U,E)$ corresponds to a (co-)CR quaternionic vector space then\/ $\mathcal{U}$ is its holomorphic vector bundle, introduced in \cite{fq_1}\,.
In fact, $(U,E)$ corresponds to a co-CR quaternionic vector space if and only if\/ $\mathcal{U}$ is a holomorphic vector bundle
and $\mathcal{U}=\mathcal{U}_+$\,. Dually, $(U,E)$ corresponds to a CR quaternionic vector space if and only if\/ $\mathcal{U}=\mathcal{U}_-$
(note that, $\mathcal{U}_-$ is a holomorphic vector bundle for any pair).
See \cite{vq} for more information on the functor $(U,E)\mapsto\mathcal{U}$\,.\\
\indent
Here are the basic examples of (co-)CR quaternionic vector spaces.

\begin{exm}[cf.\ \cite{fq_1}\,] \label{exm:V_k}
1) Let $V_k$\,, $(k\geq1)$\,, be the vector subspace of $\Hq^{\!k}$ formed of all vectors of the form
$(z_1\,,\overline{z_1}+z_2\,{\rm j}\,, z_3-\overline{z_2}\,{\rm j}\,,\ldots)\,,$
where $z_1\,,\ldots,z_k$ are complex numbers and $\overline{z_k}=(-1)^kz_k$\,. Then $(V_k,\Hq^{\!k})$ corresponds to a
co-CR quaternionic vector space and its holomorphic vector bundle is $\ol(2k)$\,. Hence, the dual pair
is a CR quaternionic vector space and its holomorphic vector bundle is $\ol(-2k)$\,.\\
\indent
2) Let $V'_0=\{0\}$ and, for $k\geq1$\,, let $V'_k$ be the vector subspace of $\Hq^{\!2k+1}$ formed of all vectors of the form
$(z_1\,,\overline{z_1}+z_2\,{\rm j}\,, z_3-\overline{z_2}\,{\rm j}\,,\ldots, \overline{z_{2k-1}}+z_{2k}\,{\rm j}\,,-\overline{z_{2k}}\,{\rm j})\,,$
where $z_1\,,\ldots,z_{2k}$ are complex numbers. Then $(V'_k,\Hq^{\!2k+1})$ corresponds to a
co-CR quaternionic vector space and its holomorphic vector bundle is $2\ol(2k+1)$\,. Hence, the dual pair
is a CR quaternionic vector space and its holomorphic vector bundle is $2\ol(-2k-1)$\,.
\end{exm}

\indent
Also, by \cite{fq_1}\,, any (co-)CR quaternionic vector space is isomorphic to a product, unique up to the order of factors,
in which each factor is given by Example \ref{exm:V_k}(1) or~(2)\,.

\begin{defn}
A \emph{linear $f$-quaternionic structure} on a vector space $U$ is a pair $(E,V)$\,, where
$E$ is a quaternionic vector space such that $U,V\subseteq E$, $E=U\oplus V$ and $J(V)\subseteq U$, for any $J\in Z$.\\
\indent
An \emph{$f$-quaternionic vector space} is a vector space endowed with a linear $f$-quaternionic structure.
\end{defn}

\indent
Let $(U,E,V)$ be an $f$-quaternionic vector space; denote by $\iota:U\to E$ the inclusion
and by $\r:E\to U$ the projection determined by the decomposition $E=U\oplus V$.\\
\indent
Then $(E,\iota)$ and $(E,\r)$ are linear CR-quaternionic and co-CR quaternionic structures, respectively, which are \emph{compatible}.\\
\indent
The \emph{$f$-quaternionic linear maps} are defined, accordingly, by using the compatible linear CR and co-CR quaternionic
structures determining a linear $f$-quaternionic structure.\\
\indent
{}From any $f$-quaternionic vector space $(U,E,V)$\,, with $\dim\!E=4k$, $\dim\!V=l$\,,
there exists an $f$-quaternionic linear isomorphism onto $({\rm Im}\Hq\!)^l\times\Hq^{\!4k-l}$
(this follows, for example, from the classification of (co-)CR quaternionic vector spaces \cite{fq_1}\,).\\
\indent
We end this section with the description of the Lie group $G$ of $f$-quaternionic linear isomorphisms
of $({\rm Im}\Hq\!)^l\times\Hq^{\!m}$. For this, let $\r_k:{\rm Sp}(1)\cdot{\rm GL}(k,\Hq)\to{\rm SO}(3)$
be the Lie group morphism defined by $\r_k(q\cdot A)=\pm q$\,, for any
$q\cdot A\in{\rm Sp}(1)\cdot{\rm GL}(k,\Hq)$\,, $(k\geq1)$\,. Denote
$$H=\bigl\{(A,A')\in\bigl({\rm Sp}(1)\cdot{\rm GL}(l,\Hq)\bigr)\times\bigl({\rm Sp}(1)\cdot{\rm GL}(m,\Hq)\bigr)\,|\,
\r_l(A)=\r_m(A')\,\bigr\}\;.$$
Then $H$ is a closed subgroup of ${\rm Sp}(1)\cdot{\rm GL}(l+m,\Hq)$ and $G$ is the closed subgroup of $H$
formed of those elements $(A,A')\in H$ such that $A$ preserves $\R^l\subseteq\Hq^{\!l}$. This follows
from the fact that there are no nontrivial $f$-quaternionic linear maps from ${\rm Im}\Hq$ to $\Hq$
(and from $\Hq$ to ${\rm Im}\Hq$). Now, the canonical basis of ${\rm Im}\Hq$ induces a linear isomorphism
$({\rm Im}\Hq\!)^l=\bigl(\R^l\bigr)^{\!3}$
and, therefore, an effective action $\s$ of ${\rm GL}(l,\R)$ on $({\rm Im}\Hq\!)^l$.
We define an effective action of ${\rm GL}(l,\R)\times\bigl({\rm Sp}(1)\cdot{\rm GL}(m,\Hq)\bigr)$
on $({\rm Im}\Hq\!)^l\times\Hq^{\!m}$ by
$$(A,q\cdot B)(X,Y)=\bigl(q\bigl(\s(A)(X)\bigr)q^{-1},q\,YB^{-1}\bigr)\;,$$
for any $A\in{\rm GL}(l,\R)$\,, $q\cdot B\in{\rm Sp}(1)\cdot{\rm GL}(m,\Hq)$\,, $X\in({\rm Im}\Hq\!)^l$
and $Y\in\Hq^{\!m}$.

\begin{prop} \label{prop:f-q_group}
There exists an isomorphism of Lie groups
$$G={\rm GL}(l,\R)\times\bigl({\rm Sp}(1)\cdot{\rm GL}(m,\Hq)\bigr)\;,$$
given by $(A,A')\mapsto(A|_{\R^l},A')$\,, for any $(A,A')\in G$.\\
\indent
In particular, the group of $f$-quaternionic linear isomorphisms of $({\rm Im}\Hq\!)^l$ is isomorphic to
${\rm GL}(l,\R)\times{\rm SO}(3)$\,.
\end{prop}

\indent
Note that, the group of $f$-quaternionic linear isomorphisms of ${\rm Im}\Hq\!$ is ${\rm CO}(3)$\,.

\section{A few basic facts on CR quaternionic manifolds} \label{section:cr_q}

\indent
In this section we recall, for the reader's convenience, a few basic facts on CR quaternionic manifolds
(we refer to \cite{fq_1} for further details).\\
\indent
A (smooth) \emph{bundle of associative algebras} is a vector bundle whose typical fibre
is a (finite-dimensional) associative algebra and whose structural group is the group of automorphisms
of the typical fibre. Let $A$ and $B$ be bundles of associative algebras. A morphism of vector bundles $\r:A\to B$ is called
a \emph{morphism of bundles of associative algebras} if $\r$ restricted to each fibre is a morphism of
associative algebras.\\
\indent
Recall that a \emph{quaternionic vector bundle} over a manifold $M$ is a real vector bundle $E$ over $M$ endowed
with a pair $(A,\r)$ where $A$ is a bundle of associative algebras, over $M$, with typical fibre
$\Hq$ and $\r:A\to{\rm End}(E)$ is a morphism of bundles of associative algebras; we say that $(A,\r)$
is a \emph{linear quaternionic structure on $E$} (see~\cite{Bon}\,). Standard arguments (see \cite{IMOP}\,)
apply to show that a quaternionic vector bundle of (real) rank $4k$ is just
a (real) vector bundle endowed with a reduction of its structural group to ${\rm Sp}(1)\cdot{\rm GL}(k,\Hq)$\,.\\
\indent
If $(A,\r)$ defines a linear quaternionic structure on a vector bundle $E$ then we denote $Q=\r({\rm Im}A)$\,, and
by $Z$ the sphere bundle of $Q$.\\
\indent
Recall \cite{Sal-dg_qm} (see \cite{IMOP}\,) that, a manifold is \emph{almost quaternionic} if and only if its
tangent bundle is endowed with a linear quaternionic structure.

\begin{defn} \label{defn:almost_cr_q}
Let $E$ be a quaternionic vector bundle on a manifold $M$ and let $\iota:TM\to E$ be an injective
morphism of vector bundles. We say that $(E,\iota)$ is an \emph{almost CR quaternionic structure} on $M$ if $(E_x,\iota_x)$ is a linear
CR quaternionic structure on $T_xM$, for any $x\in M$.\\
\indent
An \emph{almost CR quaternionic manifold} is a manifold endowed with an almost CR quaternionic structure.
\end{defn}

\indent
On any almost CR quaternionic manifold $(M,E,\iota)$ for which $E$ is endowed with a connection $\nabla$,
compatible with its linear quaternionic structure, there can be
defined a natural almost twistorial structure, as follows.
For any $J\in Z$, let $\mathcal{B}_J\subseteq T^{\C}_J\!Z$ be the horizontal lift, with respect to $\nabla$,
of $\iota^{-1}\bigl(E^J\bigr)$, where $E^J\subseteq E^{\C}_{\p(J)}$ is the eigenspace of $J$ corresponding
to $-{\rm i}$\,. Define $\Cal_J=\mathcal{B}_J\oplus({\rm ker}\dif\!\p)^{0,1}_J$, $(J\in Z)$\,.
Then $\Cal$ is an almost CR structure on $Z$ and $(Z,M,\p,\Cal)$ is
\emph{the almost twistorial structure of $(M,E,\iota,\nabla)$}.

\begin{defn} \label{defn:cr_q}
An \emph{(integrable almost) CR quaternionic structure} on $M$ is a triple $(E,\iota,\nabla)$\,,
where $(E,\iota)$ is an almost CR quaternionic structure on $M$ and $\nabla$ is an almost quaternionic connection
of $(M,E,\iota)$ such that the almost twistorial structure of $(M,E,\iota,\nabla)$ is integrable
(that is, $\Cal$ is integrable). Then $(M,E,\iota,\nabla)$ is a \emph{CR quaternionic manifold} and
the CR manifold $(Z,C)$ is its \emph{twistor space}.
\end{defn}

\indent
A main source of CR quaternionic manifolds is provided by the submanifolds of quaternionic manifolds.

\begin{defn} \label{defn:realiz_cr_q}
Let $(M,E,\iota,\nabla)$ be a CR quaternionic manifold and let $(Z,\Cal)$ be its twistor space. We say that $(M,E,\iota,\nabla)$ is \emph{realizable} if $M$ is an embedded submanifold of a quaternionic manifold $N$
such that $E=TN|_M$, as quaternionic vector bundles, and $\Cal=T^{\C}\!Z\cap(T^{0,1}Z_N)|_M$, where $Z_N$ is the twistor space of $N$.\\
\indent
Then $N$ is \emph{the heaven space} of $(M,E,\iota,\nabla)$\,.
\end{defn}

\indent
By \cite[Corollary 5.4]{fq_1}\,, any real-analytic CR quaternionic manifold is realizable.

\section{Co-CR quaternionic manifolds} \label{section:co-cr_q}

\indent
An \emph{almost co-CR structure} on a manifold $M$ is a complex vector subbundle $\Cal$ of $T^{\C\!}M$ such that
$\Cal+\overline{\Cal}=T^{\C\!}M$. An \emph{(integrable almost) co-CR structure} is an almost co-CR structure whose
space of sections is closed under the bracket.\\
\indent
Note that, if $\phi:M\to(N,J)$ is a submersion onto a complex manifold then $(\dif\!\phi)^{-1}\bigl(T^{0,1}N\bigr)$
is a co-CR structure on $M$; moreover, any co-CR structure is, locally, of this form.

\begin{defn} \label{defn:almost_co-cr_q}
Let $E$ be a quaternionic vector bundle on a manifold $M$ and let $\r:E\to TM$ be a surjective morphism
of vector bundles. Then $(E,\r)$ is called an \emph{almost co-CR quaternionic structure}, on $M$, if $(E_x,\r_x)$ is a
linear co-CR quaternionic structure on $T_xM$, for any $x\in M$.
If, further, $E$ is a hypercomplex vector bundle then $(E,\r)$ is called an \emph{almost hyper-co-CR structure} on $M$. An \emph{almost co-CR quaternionic manifold} (\emph{almost hyper-co-CR manifold}) is a manifold endowed with
an almost co-CR quaternionic structure (almost hyper-co-CR structure).
\end{defn}

\indent
Any almost co-CR quaternionic (hyper-co-CR) structure $(E,\r)$ for which $\r$ is an isomorphism is an almost
quaternionic (hypercomplex) structure.

\begin{exm} \label{exm:3d_almost}
Let $(M,c)$ be a three-dimensional conformal manifold and let $L=\bigl(\Lambda^3TM\bigr)^{1/3}$ be the line bundle of $M$.
Then, $E=L\oplus TM$ is an oriented vector bundle of rank four endowed
with a (linear) conformal structure such that $L=(TM)^{\perp}$. Therefore $E$ is a quaternionic vector bundle
and $(M,E,\r)$ is an almost co-CR quaternionic manifold, where $\r:E\to TM$ is the projection. Moreover, any
three-dimensional almost co-CR quaternionic manifold is obtained this way.
\end{exm}

\indent
Next, we are going to introduce a natural almost twistorial structure
(see \cite{LouPan-II} for the definition of almost twistorial structures)
on any almost co-CR quaternionic manifold
$(M,E,\r)$ for which $E$ is endowed with a connection $\nabla$ compatible with its linear quaternionic structure.\\
\indent
For any $J\in Z$, let $\Cal_J\subseteq T^{\C}_J\!Z$ be the direct sum of $({\rm ker}\dif\!\p)_J^{0,1}$
and the horizontal lift, with respect to $\nabla$, of $\r(E^J)$, where $E^J$ is the eigenspace of $J$
corresponding to $-{\rm i}$\,. Then $\Cal$ is an almost co-CR structure on $Z$ and
$(Z,M,\p,\Cal)$ is \emph{the almost twistorial structure of $(M,E,\r,\nabla)$}.\\
\indent
The following definition is motivated by \cite[Remark 2.10(2)\,]{IMOP}\,.

\begin{defn} \label{defn:co-cr_q}
An \emph{(integrable almost) co-CR quaternionic manifold} is an almost
\mbox{co-CR quaternionic} manifold $(M,E,\r)$ endowed with a compatible connection $\nabla$ on $E$ such that
the associated almost twistorial structure $(Z,M,\p,\Cal)$ is integrable (that is, $\Cal$ is integrable).
If, further, $E$ is a hypercomplex vector bundle and the connection induced by $\nabla$ on $Z$ is trivial then
$(M,E,\r,\nabla)$ is an \emph{(integrable almost) hyper-co-CR manifold}.
\end{defn}

\begin{exm}
Let $(M,c)$ be a three-dimensional conformal manifold and let $(E,\r)$ be the corresponding almost co-CR structure,
where $E=L\oplus TM$ with $L$ the line bundle of $M$. Let $D$ be a Weyl connection on $(M,c)$ and
let $\nabla=D^L\oplus D$\,, where $D^L$ is the connection induced by $D$ on $L$.
It follows that $(M,E,\r,\nabla)$ is co-CR quaternionic if and only if $(M,c,D)$ is Einstein--Weyl
(that is, the trace-free symmetric part of the Ricci tensor of $D$ is zero).\\
\indent
Furthermore, let $\mu$ be a section of $L^*$ such that the connection defined by $$D^{\mu}_{\,X}Y=D_XY+\mu\,X\times_cY\;$$
for any vector fields $X$ and $Y$ on $M$, induces a flat connection on $L^*\otimes TM$.
Then $(M,E,\iota,\nabla^{\mu})$ is, locally, a hyper-co-CR manifold,
where $\nabla^{\mu}=(D^{\mu})^L\oplus D^{\mu}$\,, with $(D^{\mu})^L$ the connection
induced by $D^{\mu}$ on $L$ (this follows from well-known results; see \cite{LouPan-II} and the references therein).
\end{exm}

\indent
Let $\t=(Z,M,\p,\Cal)$ be the twistorial structure of a co-CR quaternionic manifold
$(M,E,\r,\nabla)$\,. Recall \cite{LouPan-II} that $\t$ is \emph{simple} if and only if $\Cal\cap\overline{\Cal}$ is a simple foliation
(that is, its leaves are the fibres of a submersion) whose leaves intersect each fibre of $\p$ at most once.
Then $\bigl(T,\dif\!\phi(\Cal)\bigr)$ is the \emph{twistor space} of $\t$, where $\phi:Z\to T$ is the submersion
whose fibres are the leaves of $\Cal\cap\overline{\Cal}$.

\begin{exm}
Any co-CR quaternionic vector space is a co-CR quaternionic manifold, in an obvious way; moreover,
the associated twistorial structure is simple and its twistor space is just its holomorphic vector bundle.
\end{exm}

\begin{thm} \label{thm:co-cr_q}
Let $(M,E,\r,\nabla)$ be a co-CR quaternionic manifold, $\rank E=4k$\,, $\rank({\rm ker}\,\r)=l$\,. If the twistorial structure of $(M,E,\r,\nabla)$ is simple then it is real analytic and its twistor space
is a complex manifold of dimension $2k-l+1$ endowed with a locally complete family of complex projective lines $\{Z_x\}_{x\in M^{\C}}$.
Furthermore, for any $x\in M$, the normal bundle of the corresponding twistor line $Z_x$ is the holomorphic vector bundle of $(T_xM,E_x,\r_x)$\,.
\end{thm}
\begin{proof}
Let $(Z,M,\p,\Cal)$  be the twistorial structure of $(M,E,\r,\nabla)$\,. Let $\phi:Z\to T$ be the submersion whose fibres are the leaves of
$\Cal\cap\overline{\Cal}$. Obviously, $\dif\!\phi(\Cal)$ defines a complex structure on $T$ of dimension $2k-l+1$\,.
Furthermore, if for any $x\in M$ we denote $Z_x=\phi(\p^{-1}(x))$ then $Z_x$ is a complex submanifold of $T$
whose normal bundle is the holomorphic vector bundle of $(T_xM,E_x,\r_x)$\,. The proof follows from \cite{Kod} and
\cite[Proposition 2.5]{Ro-LeB_nonrealiz}\,.
\end{proof}

\begin{prop}
Let $(M,E,\r,\nabla)$ be a co-CR quaternionic manifold whose twistorial structure is simple; denote by $\phi:Z\to T$
the corresponding holomorphic submersion onto its twistor space. Then $(M,E,\r,\nabla)$ is hyper-co-CR if and only if there exists a surjective holomorphic submersion $\psi:T\to\C\!P^1$
such that the fibres of $\psi\circ\phi$ are integral manifolds of the connection induced by\/ $\nabla$ on $Z$\,.
\end{prop}
\begin{proof}
Denote by $\H$ the connection induced by $\nabla$ on $Z$\,. Then $\H$ is integrable if and only if $\dif\!\phi(\H)$ is
a holomorphic foliation on $T$; furthermore, this foliation is simple if and only if $E$ is hypercomplex and $\H$ is the
trivial connection on~$Z$\,.
\end{proof}

\section{$f$-Quaternionic manifolds} \label{section:f_q}

\indent
Let $F$ be an almost $f$-structure on a manifold $M$; that is, $F$ is a field of endomorphisms of $TM$ such that $F^3+F=0$\,.
Denote by $\Cal$ the eigenspace of $F$ with respect to $-{\rm i}$ and let $\mathcal{D}=\Cal\oplus{\rm ker}F$. Then
$\Cal$ and $\mathcal{D}$ are \emph{compatible} almost CR and almost co-CR structures, respectively. An \emph{(integrable almost) $f$-structure} is an almost $f$-structure for which the corresponding almost CR and almost co-CR
structures are integrable.

\begin{defn} \label{defn:f-q}
An \emph{almost $f$-quaternionic structure} on a manifold $M$ is a pair $(E,V)$\,,
where $E$ is a quaternionic vector bundle on $M$ and $TM$ and $V$ are vector subbundles of $E$
such that $E=TM\oplus V$ and $J(V)\subseteq TM$, for any $J\in Z$. An \emph{almost hyper-$f$-structure} on a manifold $M$ is an almost $f$-quaternionic structure $(E,V)$ on $M$
such that $E$ is a hypercomplex vector bundle. An \emph{almost $f$-quaternionic manifold} (\emph{almost hyper-$f$-manifold}) is a manifold endowed with
an almost $f$-quaternionic structure (almost hyper-$f$-structure).
\end{defn}

\indent
With the same notations as in Definition \ref{defn:f-q}\,, an almost $f$-quaternionic structure
(almost hyper-$f$-structure) for which $V$ is the zero bundle is an almost quaternionic structure
(almost hypercomplex structure).\\
\indent
Let $k$ and $l$ be positive integers, $k\geq l$, and denote by $G_{k,l}$ the group of $f$-quaternionic
linear isomorphisms of $({\rm Im}\Hq)^l\times\Hq^{k-l}$. The next result is an immediate consequence of the description of $G_{k,l}$
given in  Section~\ref{section:review_lin_co-CR_q}\,.

\begin{prop} \label{prop:almost_f-q_as_a_reduction}
Let $M$ be a manifold of dimension $4k-l$. Then any almost $f$-quaternionic structure $(E,V)$ on $M$,
with $\rank E=4k$ and $\rank V=l$\,, corresponds to a reduction of the frame bundle of $M$ to $G_{k,l}$\,.\\
\indent
Furthermore, if $(P,M,G_{k,l})$ is the reduction of the frame bundle of $M$, corresponding to $(E,V)$\,,
then $V$ is the vector bundle associated to $P$ through the canonical morphism of Lie groups $G_{k,l}\to{\rm GL}(l,\R)$\,.
\end{prop}

\begin{exm} \label{exm:first-fq's}
1) A three-dimensional almost $f$-quaternionic manifold is just a (three-dimensional) conformal manifold.\\
\indent
2) Let $N$ be an almost quaternionic manifold endowed with a Hermitian metric and let $M$ be a hypersurface in $N$.
Then $\bigl(TN|_M,(TM)^{\perp}\bigr)$ is an almost $f$-quaternionic structure on $M$.
\end{exm}

\indent
Obviously, any almost $f$-quaternionic structures $(E,V)$ on a manifold $M$ corresponds to a pair
$(E,\iota)$ and $(E,\r)$ of almost CR quaternionic and co-CR quaternionic structures on $M$,
where $\iota:TM\to E$ and $\r:E\to TM$ are the inclusion and projection, respectively.

\begin{defn}
Let $(M,E,V)$ be an almost $f$-quaternionic manifold. Let $(E,\iota)$ and $(E,\r)$ be the
almost CR quaternionic and co-CR quaternionic structures, respectively, corresponding to $(E,V)$\,.
Let $\nabla$ be a connection on $E$ compatible with its linear quaternionic structure and let
$\t$ and $\t_c$ be the almost twistorial structures of $(M,E,\iota,\nabla)$
and $(M,E,\r,\nabla)$\,, respectively. We say that $(M,E,V,\nabla)$ is an \emph{$f$-quaternionic manifold} if the almost twistorial
structures $\t$ and $\t_c$ are integrable. If, further, $E$ is hypercomplex and $\nabla$ induces
the trivial flat connection on $Z$ then $(M,E,V,\nabla)$ is an \emph{(integrable almost) hyper-$f$-manifold}.\\
\indent
Let $(M,E,V,\nabla)$ be an $f$-quaternionic manifold and let $Z$ and $Z_c$ be the twistor spaces
of $\t$ and $\t_c$\,, respectively (we assume, for simplicity, that $\t_c$ is simple).
Then $Z$ is called the \emph{CR twistor space} and $Z_c$ is called the \emph{twistor space} of $(M,E,V,\nabla)$\,.
\end{defn}

\indent
Let $(M,E,V)$ be an almost $f$-quaternionic manifold and let $\nabla$ be a connection on $E$ compatible
with its linear quaternionic structure. Let $\Cal$ and $\mathcal{D}$ be the almost CR and almost co-CR structures
on $Z$ determined by $\nabla$ and the underlying almost CR quaternionic and almost co-CR quaternionic structures
of $(M,E,V)$\,, respectively. Then $\Cal$ and $\mathcal{D}$ are compatible; therefore $(M,E,V,\nabla)$ is $f$-quaternionic if and only if
the corresponding almost $f$-structure on $Z$ is integrable.\\
\indent
Let $(M,E,V)$ be an almost $f$-quaternionic manifold, $\rank E=4k$\,, $\rank V=l$\,, and $D$ some compatible connection on $M$
(equivalently, $D$ is a linear connection on $M$ which corresponds to a principal connection on the reduction to $G_{k,l}$\,,
of the frame bundle of $M$, corresponding to $(E,V)$\,). Then $D$ induces a connection $D^V$ on $V$.
Moreover, $\nabla=D^V\oplus D$ is compatible with the linear quaternionic structure on $E$.

\begin{cor} \label{cor:f-q}
Let $(M,E,V,\nabla)$ be an $f$-quaternionic manifold, $\rank E=4k$\,, $\rank V=l$\,, where $\nabla=D^V\oplus D$
for some compatible connection $D$ on $M$. Denote by $\t$ and $\t_c$ the associated twistorial structures. Then, locally, the twistor space of $(M,\t_c)$ is a complex manifold, of complex dimension $2k-l+1$\,, endowed
with a locally complete family of complex projective lines each of which has normal bundle
$2(k-l)\ol(1)\oplus l\,\ol(2)$\,.\\
\indent
Furthermore, if $(M,E,V,\nabla)$ is real analytic then, locally, there exists a twistorial map
from the corresponding heaven space $N$, endowed with its twistorial structure, to $(M,\t_c)$ which is a retraction
of the inclusion $M\subseteq N$.
\end{cor}
\begin{proof}
By passing to a convex open set of $D$, if necessary, we may suppose that $\t_c$ is simple. Thus, the first assertion
is a consequence of Theorem \ref{thm:co-cr_q}\,. The second statement follows from the fact that there exists a
holomorphic submersion from the twistor space of $N$,
endowed with its twistorial structure, to the twistor space of $(M,\t_c)$\,, which maps diffeomorphically twistor lines
onto twistor lines.
\end{proof}

\indent
Note that, if $\dim M=3$ then Corollary \ref{cor:f-q} gives results of \cite{LeB-Hspace} and \cite{Hit-complexmfds}\,.

\begin{exm} \label{exm:Gr_3^+}
Let $M^{3l}={\rm Gr}_3^+(l+3,\R)$ be the Grassmann manifold of oriented vector subspaces of dimension $3$
of $\R^{l+3}$, $(l\geq1)$\,. Alternatively, $M^{3l}$ can be defined as the Riemannian symmetric space
${\rm SO}(l+3)/\bigl({\rm SO}(l)\times{\rm SO}(3)\bigr)$\,. As the structural group of the frame bundle of
$M^{3l}$ is ${\rm SO}(l)\times{\rm SO}(3)$\,,
from Proposition \ref{prop:almost_f-q_as_a_reduction} we obtain that $M^{3l}$ is canonically endowed with an almost
$f$-quaternionic structure. Moreover, if we endow $M^{3l}$ with its Levi-Civita connection then we obtain an
$f$-quaternionic manifold. Its twistor space is the hyperquadric $Q_{l+1}$ of isotropic
one-dimensional complex vector subspaces of $\C^{\!l+3}$, considered as the complexification of the
(real) Euclidean space of dimension $l+3$\,. Further, the CR twistor space $Z$ of $M^{3l}$
can be described as the closed submanifold of $Q_{l+1}\times M^{3l}$
formed of those pairs $(\ell,p)$ such that $\ell\subseteq p^{\C}$. Under the orthogonal decomposition $\R^{l+4}=\R\oplus\R^{l+3}$, we can embed
$M^{3l}$ as a totally geodesic submanifold of the quaternionic manifold
$\widetilde{M}^{4l}={\rm Gr}_4^+(l+4,\R)$ as follows: $p\mapsto\R\oplus p$\,, $(p\in M^{3l})$\,.
Recall (see \cite{LeB-twist_qK}\,) that the twistor space of $\widetilde{M}^{4l}$ is the manifold
$\widetilde{Z}={\rm Gr}_2^0(l+4,\C)$ of isotropic complex vector subspaces of dimension $2$ of $\C^{\!l+4}$,
where the projection $\widetilde{Z}\to\widetilde{M}$ is given by $q\mapsto p$\,, with $q$ a self-dual subspace
of $p^{\C}$ (in particular, $p^{\C}=q\oplus\overline{q}$). Consequently, the CR twistor space $Z$ of $M^{3l}$
can be embedded in $\widetilde{Z}$ as follows:
$(\ell,p)\mapsto q$\,, where $q$ is the unique self-dual subspace of $(\R\oplus p)^{\C}$ which intersects
$p^{\C}$ along $\ell$\,.\\
\indent
In the particular case $l=1$ we obtain the well-known fact (see \cite{BaiWoo2}\,) that the twistor space of $S^3$
is $Q_2\,\bigl(=\C\!P^1\times\C\!P^1\bigr)$\,. Also, the CR twistor space of $S^3$ can be identified with the sphere
bundle of $\ol(1)\oplus\ol(1)$\,. Similarly, the dual of $M^{3l}$ is, canonically, an $f$-quaternionic manifold whose twistor space
is an open set of $Q_{l+1}$\,.
\end{exm}

\begin{exm} \label{exm:Gr_2^0symplectic}
Let ${\rm Gr}_2^0(2n+2,\C\!)$ be the complex hypersurface of the Grassmannian ${\rm Gr}_2(2n+2,\C\!)$ of two-dimensional complex vector subspaces of
$\C^{\!2n+2}\bigl(=\Hq^{\!n+1}\bigr)$ formed of those $q\in{\rm Gr}_2(2n+2,\C\!)$ which are isotropic with respect to the underlying complex
symplectic structure $\o$ of $\C^{\!2n+2}$; note that,
$${\rm Gr}_2^0(2n+2,\C\!)={\rm Sp}(n+1)/\bigl({\rm U}(2)\times{\rm Sp}(n-1)\bigr)\;.$$
\indent
Then ${\rm Gr}_2^0(2n+2,\C\!)$ is a real-analytic $f$-quaternionic manifold and its heaven space is ${\rm Gr}_2(2n+2,\C\!)$.
Its twistor space is ${\rm Gr}_2^0(2n+2,\C\!)$ itself, considered as a complex manifold.\\
\indent
To describe the CR twistor space of ${\rm Gr}_2^0(2n+2,\C\!)$, firstly, recall
that the twistor space of ${\rm Gr}_2(2n+2,\C\!)$ is the flag manifold ${\rm F}_{1,2n+1}(2n+2,\C\!)$
formed of the pairs $(\ell,p)$ with $\ell$ and $p$ complex vector subspaces of $\C^{\!2n+2}$ of dimensions $1$ and $2n+1$\,, respectively,
such that $\ell\subseteq p$\,.\\
\indent
Now, let $Z\subseteq {\rm Gr}_2^0(2n+2,\C\!)\times {\rm Gr}_2^0(2n+2,\C\!)$ be formed of the pairs $(p,q)$ such that
$p\cap q$ and $p\cap q^{\perp}$ are nontrivial and the latter is
contained by the kernel of $\o|_{q^{\perp}}$, where the orthogonal complement is taken with respect to the underlying Hermitian metric of $\C^{\!2n+2}$.
Then the embedding $Z\to{\rm F}_{1,2n+1}(2n+2,\C\!)$\,, $(p,q)\mapsto(p\cap q, q^{\perp}+p\cap q)$ induces a CR structure with respect to which
$Z$ is the CR twistor space of ${\rm Gr}_2^0(2n+2,\C\!)$.\\
\indent
Note that, if $n=1$ we obtain the $f$-quaternionic manifold of Example \ref{exm:Gr_3^+} with $l=2$\,.
\end{exm}

\indent
The next example is related to a construction of \cite{Swa-bundle} (see, also, \cite[Example 4.4]{IMOP}\,).

\begin{exm} \label{exm:SO(Q)}
Let $M$ be a quaternionic manifold, $\nabla$ a quaternionic connection on it and $Z$ its twistor space.\\
\indent
Then $Z$ is the sphere bundle of an oriented Riemannian vector bundle of rank three $Q$.
By extending the structural group of the frame bundle $\bigl({\rm SO}(Q),M,{\rm SO}(3,\R)\bigr)$ of $Q$ we obtain a principal bundle
$\bigl(H,M,\Hq^{\!*}\!/{\mathbb{Z}_2}\bigr)$\,.\\
\indent
Let $q\in S^2\,(\subseteq{\rm Im}\Hq)$\,. The morphism of Lie groups $\C^{\!*}\to\Hq^{\!*}$, $a+b\,{\rm i}\mapsto a-bq$ induces an
action of $\C^{\!*}$ on $H$ whose quotient space is $Z$ (considered with its underlying smooth structure); denote by $\psi_q:H\to Z$
the projection. Moreover, $(H,Z,\C^{\!*})$ is a principal bundle on which $\nabla$ induces a principal connection for which the $(0,2)$ component
of its curvature form is zero. Therefore the complex structures of $Z$ and of the fibres of $H$ induce, through this connection,
a complex structure $J_q$ on $H$.\\
\indent
We, thus, obtain a hypercomplex manifold $(H,J_{\rm i},J_{\rm j},J_{\rm k})$ which is the heaven space of an $f$-quaternionic structure
on ${\rm SO}(Q)$ (in fact, a hyper-$f$ structure). Note that, the twistor space of ${\rm SO}(Q)$ is $\C\!P^1\times Z$ and the corresponding
projection from $S^2\times{\rm SO}(Q)$ onto $\C\!P^1\times Z$ is given by $(q,u)\mapsto\bigl(q,\psi_q(u)\bigr)$\,,
for any $(q,u)\in S^2\times{\rm SO}(Q)$\,.\\
\indent
If $M=\Hq\!P^k$ then the factorisation through $\mathbb{Z}_2$ is unnecessary and we obtain an $f$-quaternionic structure on $S^{4k+3}$
with heaven space $\Hq^{\!k+1}\setminus\{0\}$ and twistor space $\C\!P^1\times\C\!P^{2k+1}$.
\end{exm}

\indent
Let $(M,E,V)$ be an almost $f$-quaternionic manifold, with $\rank V=l$\,,
and $(P,M,G_{k,l})$ the corresponding reduction of the frame bundle of $M$, where $\rank E=4k$\,.
Then $TM=(V\otimes Q)\oplus W$, where $W$ is
the quaternionic vector bundle associated to $P$ through the canonical morphisms of Lie groups
$G_{k,l}\longrightarrow{\rm Sp}(1)\cdot{\rm GL}(k-l,\Hq)$\,. Note that, $W$ is the largest quaternionic
vector subbundle of $E$ contained by $TM$.

\begin{thm} \label{thm:f-q}
Let $(M,E,V)$ be an almost $f$-quaternionic manifold and let $D$ be a compatible torsion free connection, $\rank E=4k$\,,
$\rank V=l$\,; suppose that $(k,l)\neq(2,2)\,,(1,0)$\,. Then $(M,E,V,\nabla)$ is $f$-quaternionic, where $\nabla=D^V\oplus D$.
Moreover, $W$ is integrable and geodesic,
with respect to $D$ (equivalently, $D_XY$ is a section of $W$, for any sections $X$ and\/ $Y$ of $W$).
\end{thm}
\begin{proof}
Let $\iota:TM\to E$ be the inclusion and $\r:E\to TM$ the projection.
It quickly follows that
we may apply \cite[Theorem 4.6]{fq_1} to obtain that $(M,E,\iota,\nabla)$ is CR quaternionic.
To prove that $(M,E,\r,\nabla)$ is co-CR quaternionic we apply \cite[Theorem A.3]{fq_1} to $D$. Thus, we obtain that
it is sufficient to show that for any $J\in Z$ and any $X,Y,Z\in E^J$
we have $R^D(\r(X),\r(Y))(\r(Z))\in\r(E^J)$\,, where $E^J$ is the eigenspace of $J$, with respect to $-{\rm i}$\,,
and $R^D$ is the curvature form of $D$\,; equivalently, for any $J\in Z$ and any $X,Y,Z\in E^J$ we have
$R^{\nabla}(\r(X),\r(Y))Z\in E^J$, where $R^{\nabla}$ is the curvature form of $\nabla$.
The proof of the fact that $(M,E,V,\nabla)$ is $f$-quaternionic follows, similarly to the proof of \cite[Theorem 4.6]{fq_1}\,.
The last statement, follows quickly from the fact that $(\nabla_XJ)(Y)$ is a section of $W$, for any section $J$ of $Z$ and $X,\,Y$ of $W$.
\end{proof}

\indent
{}From the proof of Theorem \ref{thm:f-q} we immediately obtain the following.

\begin{cor}
Let $(M,E,V)$ be an almost $f$-quaternionic manifold and let $D$ be a compatible torsion free connection, $\rank E\geq8$\,.
Then $(M,E,\r,\nabla)$ is co-CR quaternionic, where $\r:E\to TM$ is the projection and $\nabla=D^V\oplus D$.
\end{cor}

\indent
Next, we prove two realizability results for $f$-quaternionic manifolds.

\begin{prop} \label{prop:realiz_f-q_1}
Let $(M,E,V,\nabla)$ be an $f$-quaternionic manifold, $\rank V=1$\,, where $\nabla=D^V\oplus D$ for some compatible connection
$D$ on $M$. Then $(M,E,\iota,\nabla)$ is realizable, where $\iota:TM\to E$ is the inclusion.
\end{prop}
\begin{proof}
By passing to a convex open set of $D$, if necessary, we may suppose that the twistorial structure $(Z,M,\p,\mathcal{D})$
of the co-CR quaternionic manifold $(M,E,\r)$ is simple, where $\r:E\to TM$ is the projection.
Thus, by Theorem \ref{thm:co-cr_q}\,, we have that $(Z,M,\p,\mathcal{D})$ is real analytic. It follows that $Q^{\C}$ is real analytic which, together with the relation $TM=(V\otimes Q)\oplus W$, quickly gives
that the twistorial structure $(Z,M,\p,\Cal)$ of $(M,E,\iota)$ is real analytic. By \cite[Corollary 5.4]{fq_1} the proof is complete.
\end{proof}

\indent
The next result is an immediate consequence of Theorem~\ref{thm:f-q} and Proposition~\ref{prop:realiz_f-q_1}\,.

\begin{cor} \label{cor:realiz_f-q_1}
Let $(M,E,V)$ be an almost $f$-quaternionic manifold, with $\rank V=1$\,, $\rank E\geq8$\,, and let $\nabla$ be a torsion free
connection on $E$ compatible with its linear quaternionic structure.
Then $(M,E,\iota,\nabla)$ is realizable, where $\iota:TM\to E$ is the inclusion.
\end{cor}

\indent
We end this section with the following result.

\begin{prop} \label{prop:toward_qK_heaven}
Let $(M,E,V,\nabla)$ be a real analytic $f$-quaternionic manifold, with $\rank V=1$\,, where $\nabla=D^V\oplus D$ for some
torsion free compatible connection $D$ on $M$. Let $N$ be the heaven space of $(M,E,\iota,\nabla)$\,, where $\iota:TM\to E$
is the inclusion, and denote by $Z_N$ its twistor space. Then $Z_N$ is endowed with a nonintegrable holomorphic distribution
$\H$ of codimension one, transversal to the twistor lines corresponding to the points of $N\setminus M$.
\end{prop}
\begin{proof}
By passing to a complexification, we may assume all the objects complex analytic. Furthermore, excepting $Z$,
we shall denote by the same symbols the corresponding complexifications. As for $Z$, this will denote the bundle of
isotropic directions of $Q$. Then any $p\in Z$ corresponds to a vector subspace $E^p$ of $E$. Let $\F$ be the distribution on $Z$ such that $\F_p$ is the horizontal lift, with respect to $\nabla$, of $\iota^{-1}(E^p)$\,,
$(p\in Z)$\,. As $(M,E,V,\nabla)$ is (complex) $f$-quaternionic $\F$ is integrable. Moreover, locally, we may suppose that its leaf space is $Z_N$\,. Let $\tG$ be the distribution on $Z$ such that, at each $p\in Z$, we have that $\tG_p$ is the horizontal lift
of $(V_x\otimes p^{\perp})\oplus W_x$\,, where $x=\p(p)$\,. Define $\mathscr{K}=\tG\oplus{\rm ker}\dif\!\p$\,. Then
the complex analytic versions of Cartan's structural equations
and \cite[Proposition III.2.3]{KoNo}\,, straightforwardly show that $\mathscr{K}$ is projectable with respect to $\F$.
Thus, $\mathscr{K}$ projects to a distribution $\H$ on $Z_N$ of codimension one.
Furthermore, by using again \cite[Proposition III.2.3]{KoNo}\,, we obtain that $\H$ is nonintegrable.
\end{proof}

\section{Quaternionic-K\"ahler manifolds as heaven spaces} \label{section:qK_heaven_spaces}

\indent
A \emph{quaternionic-K\"ahler} manifold is a quaternionic manifold endowed with a
(semi-Riemannian) Hermitian metric whose Levi-Civita connection is quaternionic
and whose scalar curvature is assumed nonzero.\\
\indent
Let $(M,E,\iota,\nabla)$ be a CR quaternionic manifold with $\rank E=\dim M+1$\,.
Let $W$ be the largest quaternionic vector subbundle of $E$ contained by $TM$ and denote by
$\mathcal{I}$ the (Frobenius) integrability tensor of $W$. {}From the integrability of the
almost twistorial structure of $(M,E,\iota,\nabla)$ it follows that, for any $J\in Z$\,,
the two-form $\mathcal{I}|_{E^J}$ takes values in $E^J/(E^J\cap W^{\C})$\,; as this is one-dimensional
the condition $\mathcal{I}|_{E^J}$ nondegenerate has an obvious meaning.

\begin{defn}
A CR quaternionic manifold $(M,E,\iota,\nabla)$\,, with $\rank E=\dim M+1$\,, is \emph{nondegenerate}
if $\mathcal{I}|_{E^J}$ is nondegenerate, for any $J\in Z$.
\end{defn}

\indent
Let $M$ be a submanifold of a quaternionic manifold $N$ and $Z$ the twistor space of $N$.\\
\indent
Denote by $B$ the second fundamental form of $M$ with respect to some quaternionic connection $\nabla$ on $N$;
that is, $B$ is the (symmetric) bilinear form on $M$, with values in $(TN|_M)/TM$, characterised by $B(X,Y)=\s(\nabla_XY)$\,,
for any vector fields $X$, $Y$ on $M$, where $\s:TN|_M\to(TN|_M)/TM$ is the projection.

\begin{defn} \label{defn:q_umbilic}
We say that $M$ is \emph{q-umbilical} in $N$ if for any $J\in Z|_M$ the second fundamental form of $M$
vanishes along the eigenvectors of $J$ which are tangent to $M$.
\end{defn}

\indent
{}From \cite[Propositions 1.8(ii) and 2.8]{IMOP} it quickly follows that the notion of q-umbilical submanifold,
of a quaternionic manifold, does not depend of the quaternionic connection used to define the second fundamental form.\\
\indent
Note that, if $\dim N=4$ then we retrieve the usual notion of umbilical submanifold. Also, if a quaternionic manifold
is endowed with a Hermitian metric then any umbilical submanifold of it is q-umbilical.\\
\indent
The notion of q-umbilical submanifold of a quaternionic manifold can be easily extended to CR quaternionic manifolds.
Indeed, just define the second fundamental form $B$ of $(M,E,\iota,\nabla)$ by $B(X,Y)=\tfrac12\,\s(\nabla_XY+\nabla_YX)$\,,
for any vector fields $X$ and $Y$ on $M$, where $\s:E\to E/TM$ is the projection.

\begin{thm} \label{thm:qK_heaven}
Let $N$ be the heaven space of a real analytic CR~quaternionic manifold $(M,E,\iota,\nabla)$\,, with $\rank E=\dim M+1$\,.
If $M$ is q-umbilical in $N$ then the twistor space $Z_N$ of $N$ is endowed with a nonintegrable holomorphic distribution $\H$
of codimension one, transversal to the twistor lines corresponding to the points of $N\setminus M$. Furthermore,
the following assertions are equivalent:\\
\indent
\quad{\rm (i)} $\H$ is a holomorphic contact structure on $Z_N$\,.\\
\indent
\quad{\rm (ii)} $(M,E,\iota,\nabla)$ is nondegenerate.
\end{thm}
\begin{proof}
By passing to a complexification, we may assume all the objects complex analytic. Also, we may assume $\nabla$ torsion free.
Furthermore, excepting $Z$, which will be soon described, below, we shall denote by the same symbols the corresponding complexifications.\\
\indent
Let $\dim N=4k$\,. As the complexification of ${\rm Sp}(1)\cdot{\rm GL}(k,H)$ is ${\rm SL}(2,\C)\cdot{\rm GL}(2k,\C)$,
we may assume that, locally, $TN=H\otimes F$ where $H$ and $F$ are (complex analytic) vector bundles of rank $2$ and $2k$\,, respectively.
Also, $H$ is endowed with a nowhere zero section $\ep$ of $\Lambda^2H^*$ and $\nabla=\nabla^H\otimes\nabla^F$, for some connections
$\nabla^H$ and $\nabla^F$ on $H$ and $F$, respectively, with $\nabla^H\!\ep=0$.\\
\indent
Then, by restricting to a convex neighbourhood of $\nabla$, if necessary, $Z_N$ is the leaf space of the foliation $\F_N$ on $PH$
which, at each $[u]\in PH$, is given by the horizontal lift, with respect to $\nabla^H$ of $[u]\otimes F_{\p_H(u)}$\,,
where $\p_H:H\to N$ is the projection. Let $Z=PH|_M$ and let $\F$ be the foliation induced by $\F_N$ on $Z$. Note that,
the leaf space of $\F$ is $Z_N$\,.\\
\indent
Let $PH+PF^*$ be the restriction to $N$ of $PH\times PF^*$. Then $([u],[\a])\mapsto [u]\otimes{\rm ker}\,\a$ defines an embedding
of $PH+PF^*$ into the Grassmann bundle $P$ of $(2k-1)$-dimensional vector spaces tangent to $N$. As $\nabla=\nabla^H\otimes\nabla^F$,
this embedding preserves the connections induced by $\nabla^H$, $\nabla^F$ and $\nabla$ on $PH+PF^*$ and $P$.
Let $\F_P$ be the distribution on $P$ which, at each $p\in P$, is the horizontal lift, with respect to $\nabla$, of $p\subseteq T_{\p_P(p)}N$\,,
where $\p_P:P\to N$ is the projection. Then the restriction of $\F_P$ to $PH+PF^*$ is a distribution $\F'$ on $PH+PF^*$.\\
\indent
The map $Z\to P$\,, $[u]\mapsto TM\cap\bigl([u]\otimes F_{\p_H(u)}\bigr)$\,, is an embedding whose image is contained by $PH+PF^*$.
Moreover, the fact that $M$ is q-umbilical in $N$ is equivalent to the fact that $\F$ is the restriction of $\F_P$ to $Z$.\\
\indent
If for any $([u],[\a])\in PH+PF^*$ we take the preimage of ${\rm ker}(\ep(u)\otimes\a)$ through the projection of $PH+PF^*$
we obtain a distribution of codimension one $\tG'$ on $PH+PF^*$ which contains $\F'$. Furthermore, $\tG=TZ\cap\tG'$
is a codimension one distribution on $Z$ which contains $\F$.\\
\indent
To prove that $\tG$ is projectable with respect to $\F$, firstly, observe that this is equivalent to the fact
that the integrability tensor of $\tG$ is zero when evaluated on the pairs in which one of the vectors is from $\F$.
Thus, as $\F$ is integrable, $\F=\F'|_Z$ and $\tG=TZ\cap \tG'$, it is sufficient to prove
that, at each $p\in PH+PF^*$, the integrability tensor of $\tG'$ is zero when evaluated on the pairs formed of a vector
from a basis of $\F'_p$ and a vector from a basis of a space complementary to $\F'_p$\,.\\
\indent
Let ${\rm SL}(H)$ and ${\rm GL}(F)$ be the frame bundles of $H$ and $F$, respectively, and let ${\rm SL}(H)+{\rm GL}(F)$ be the
restriction to $N$ of ${\rm SL}(H)\times {\rm GL}(F)$\,. Then the kernel of the differential of the projection
of ${\rm SL}(H)+{\rm GL}(F)$ is the trivial vector bundle over
${\rm SL}(H)+{\rm GL}(F)$ with fibre $\frak{sl}(2,\C)\oplus\frak{gl}(2k,\C)$\,.
Also, note that, for any $(u,v)\in{\rm SL}(H)+{\rm GL}(F)$\,, we have that $u\otimes v$ is a (complex-quaternionic) frame on $N$.\\
\indent
Let $G$ be the closed subgroup of ${\rm SL}(2,\C)\times{\rm GL}(2k,\C)$ which preserves some fixed pair
$\bigl([x_0],[\a_0]\bigr)\in\C\!P^1\times P\bigl(\bigl(\C^{\!2k}\bigr)^*\bigr)$\,.
Then $PH+PF^*=\bigl({\rm SL}(H)+{\rm GL}(F)\bigr)/G$ and we denote $\F''=(\dif\!\mu)^{-1}(\F')$ and $\tG''=(\dif\!\mu)^{-1}(\tG')$\,,
where $\mu$ is the projection from ${\rm SL}(H)+{\rm GL}(F)$ onto $PH+PF^*$.\\
\indent
For any $\xi\in\C^{\!2}\otimes\C^{\!2k}$ we define a horizontal vector field $B(\xi)$ which at any $(u,v)\in{\rm SL}(H)+{\rm GL}(F)$
is the horizontal lift of $(u\otimes v)(\xi)$\,. Then $\F''$ is generated by the Lie algebra of $G$ and all $B(x_0\otimes y)$ with $\a_0(y)=0$\,.
Also, $\tG''$ is generated by $\frak{sl}(2,\C)\oplus\frak{gl}(2k,\C)$ and all $B(\xi)$ with $\bigl(\ep_0(x_0)\otimes\a_0\bigr)(\xi)=0$\,,
where $\ep_0$ is the volume form on $\C^{\!2}$.\\
\indent
Further, similarly to \cite[Proposition III.2.3]{KoNo}\,, we have
$\bigl[A_1\oplus A_2,B(x_1\otimes x_2)\bigr]=B(A_1x_1\otimes x_2+x_1\otimes A_2x_2)$\,, for any $A_1\in\frak{sl}(2,\C)$\,,
$A_2\in\frak{gl}(2k,\C)$\,, $x_1\in\C^{\!2}$ and $x_2\in\C^{\!2k}$. Also, because $\nabla$ is torsion free we have that,
for any $\xi,\e\in\C^{\!2}\otimes\C^{\!2k}$, the horizontal component of $\bigl[B(\xi),B(\e)\bigr]$ is zero.
These facts quickly show that, at each $(u,v)\in{\rm SL}(H)+{\rm GL}(F)$\,,
the integrability tensor of $\tG''$ is zero when evaluated on the pairs formed of a vector
from a basis of $\F''_{(u,v)}$ and a vector from a basis of a space complementary to $\F''_{(u,v)}$\,. Consequently,
$\tG$ is projectable with respect to $\F$.\\
\indent
Next, we shall prove that $\tG$ is nonintegrable. For this, firstly, observe that those $(u,v)$ in
$\bigl({\rm SL}(H)+{\rm GL}(F)\bigr)|_M$
for which $u\otimes v$ preserves the corresponding tangent space to $M$ form a principal bundle, which we shall call
`the bundle of adapted frames', whose structural group $K$ can be described, as follows. We may write
$\C^{\!2}\otimes\C^{\!2k}=\frak{gl}(2,\C)\oplus\bigl(\C^{\!2}\otimes\C^{\!2k-2}\bigr)$ so that $K$ is the closed subgroup of
${\rm SL}(2,\C)\times{\rm GL}(2k,\C)$ which preserve ${\rm Id}_{\C^{\!2}}$. Thus, $K$ contains
${\rm SL}(2,\C)$ acting on $\frak{gl}(2,\C)\oplus\bigl(\C^{\!2}\otimes\C^{\!2k-2}\bigr)$ by $(a,(\xi,\e))\mapsto(a\xi a^{-1},\e)$\,,
for any $a\in{\rm SL}(2,\C)$\,, $\xi\in\frak{gl}(2,\C)$ and $\e\in\C^{\!2}\otimes\C^{\!2k-2}$.\\
\indent
Note that, $TM$ is the bundle associated to the bundle of adapted frames through the action of $K$ on
$\frak{sl}(2,\C)\oplus\bigl(\C^{\!2}\otimes\C^{\!2k-2}\bigr)$\,.
Also, $Z\,(\subseteq P)$ is the quotient of the bundle of adapted frames through the closed subgroup of $K$
preserving $\C\!\xi_0\oplus\bigl({\rm ker}\,\xi_0\otimes\C^{\!2k-2}\bigr)$\,, for some fixed $\xi_0\in\frak{sl}(2,\C)\setminus\{0\}$
with $\det\xi_0=0$\,.\\
\indent
If we, locally, consider a principal connection on the bundle of adapted frames then we can define, similarly to above,
the corresponding `standard horizontal vector fields' $B(\xi)$\,, for any $\xi\in\frak{sl}(2,\C)\oplus\bigl(\C^{\!2}\otimes\C^{\!2k-2}\bigr)$\,,
so that $\tG$ corresponds to the distribution generated by the Lie algebra of $K$ and all $B(\xi)$ with
$\xi\in\C^{\!2}\otimes\C^{\!2k-2}$ or $\xi\in\frak{sl}(2,\C)$ such that $\xi({\rm ker}\,\xi_0)\subseteq{\rm ker}\,\xi_0$\,.
Thus, if we take $\xi\in\frak{sl}(2,\C)$ with $\xi({\rm ker}\,\xi_0)\subseteq{\rm ker}\,\xi_0$ and $A\in\frak{sl}(2,\C)$
such that $[A,\xi]({\rm ker}\,\xi_0)\nsubseteq{\rm ker}\,\xi_0$ then $A$ and $B(\xi)$ determine
sections of $\tG$ whose bracket is not a section of $\tG$.\\
\indent
Finally, the equivalence of the assertions (i) and (ii) is a straightforward consequence of the fact that if we denote by $W$
the largest complex-quaternionic subbundle of $TN|_M$ contained by $TM$ then $\F+(\dif\!\p)^{-1}(W)=\tG$, where $\p:Z\to M$
is the projection.
\end{proof}

\indent
The next result follows immediately from \cite{LeB-twist_qK} and Theorem \ref{thm:qK_heaven}.

\begin{cor} \label{cor:qK_heaven}
The following assertions are equivalent, for a real analytic hypersurface $M$ embedded in a quaternionic manifold $N$:\\
\indent
{\rm (i)} $M$ is nondegenerate and q-umbilical.\\
\indent
{\rm (ii)} By passing, if necessary, to an open neighbourhood of $M$, there exists a metric $g$ on $N\setminus M$
such that $(N\setminus M,g)$ is quaternionic-K\"ahler and the twistor lines determined by the points of $M$
are tangent to the contact distribution, on the twistor space of $N$, corresponding to $g$\,.
\end{cor}

\indent
If $\dim M=3$ then Corollary \ref{cor:qK_heaven} and \cite[Corollary 5.5]{fq_1} give the main result of \cite{LeB-Hspace}\,.
Also, the `quaternionic contact' manifolds of \cite{Bi-qK_heaven} (see \cite{Du-qK_dim7}\,)
are nondegenerate q-umbilical CR quaternionic manifolds.

\appendix

\section{The intrinsic description of linear\\
(co-)CR quaternionic structures} \label{appendix:(co-)cr_q_intrinsic}

\indent
A \emph{conjugation}, on a quaternionic vector space, is an involutive quaternionic automorphism (not equal to the
identity); in particular, the corresponding orientation preserving isometry on the space of admissible complex structures
is a symmetry in a line.

\begin{exm}[\,\cite{Bon}\,] \label{exm:conjugation}
Let $U^{\Hq}=\Hq\otimes U$ be the \emph{quaternionification} of a vector space $U$ (the tensor product is taken over $\R$),
endowed with the linear quaternionic structure induced by the multiplication to the left.\\
\indent
If $q\in S^2$ then the association $q'\otimes u\mapsto-qq'q\otimes u$, for any $q'\in\Hq$ and $u\in U$,
defines a conjugation on $U^{\Hq}$.
\end{exm}

\indent
In fact, more can be proved.

\begin{prop}
Any pair of distinct commuting conjugations $\t_1$ and $\t_2$ on a quaternionic vector space $E$ determine
a quaternionic linear isomorphism $E=U^{\Hq}$, for some vector space $U$, so that $\t_1$ and $\t_2$ are
defined, as in Example \ref{exm:conjugation}\,, by two orthogonal imaginary unit quaternions.
\end{prop}
\begin{proof}
Let $T_1,T_2:Z\to Z$ be the orientation preserving isometries
corresponding to $\t_1$\,, $\t_2$\,, respectively, where $Z$ is the space of admissible linear complex structures on $E$.\\
\indent
As $T_1$ and $T_2$ are commuting symmetries in lines $\ell_1$ and $\ell_2$\,, respectively, it follows that
either $\ell_1=\ell_2$ or $\ell_1\perp\ell_2$\,. In the former case, we would have $T_1T_2={\rm Id}_Z$ which, together with the fact
that $\t_1$ and $\t_2$ are commuting involutions, implies $\t_1=\t_2$\,, a contradiction.
Thus, if $\ell_1$ and $\ell_2$ are generated by $I$ and $J$, respectively, then $IJ=-IJ$; denote $K=IJ$.\\
\indent
Now, $E=U^+\oplus U^-$, where $U^{\pm}={\rm ker}\bigl(\t_1\mp{\rm Id}_E\bigr)$. Furthermore, as $\t_1\t_2=\t_2\t_1$\,,
we have $U^+=V^+\oplus V^-$ and $U^-=W^+\oplus W^-$, where $V^{\pm}={\rm ker}\bigl(\t_2|_{U^+}\mp{\rm Id}_{U^+}\bigr)$
and $W^{\pm}={\rm ker}\bigl(\t_2|_{U^-}\mp{\rm Id}_{U^-}\bigr)$\,.\\
\indent
A straightforward argument shows that $IV^+=V^-$, $JV^+=W^+$ and $KV^+=W^-$. Thus, if we denote $U=V^+$ then
$E=U\oplus IU\oplus JU\oplus KU$ and the association $q\otimes u\mapsto q_0u+q_1Iu+q_2Ju+q_3Ku$\,, for any
$q=q_0+q_1{\rm i}+q_2{\rm j}+q_3{\rm k}\in\Hq$ and $u\in U$, defines a quaternionic linear isomorphism from
$U^{\Hq}$ onto $E$ which is as required.
\end{proof}

\indent
The quaternionification of a linear map is defined in the obvious way. Then a quaternionic linear map between
the quaternionifications of two vector spaces is the quaternionification of a linear map if and only if it intertwines
two distinct commuting conjugation.\\
\indent
Let $U$ be a vector space and let $\Lambda$ be the space of conjugations on~$U^{\Hq}$.\\
\indent
The next proposition reformulates a result of \cite{Bon}\,.

\begin{prop} \label{prop:q_Bonan}
There exist natural correspondences between the following:\\
\indent
\quad{\rm (i)} Linear quaternionic structures on $U$;\\
\indent
\quad{\rm (ii)} Quaternionic vector subspaces $B\subseteq U^{\Hq}$ such that $U^{\Hq}=B\oplus\sum_{\t\in\Lambda}\t(B)\,;$
\indent
\quad{\rm (iii)} Quaternionic vector subspaces $C\subseteq U^{\Hq}$ such that $U^{\Hq}=C\oplus\bigcap_{\t\in\Lambda}\t(C)\,.$
\indent
Furthermore, the correspondences are such that $C=\sum_{\t\in\Lambda}\t(B)$ and $B=\bigcap_{\t\in\Lambda}\t(C)$\,.
\end{prop}

\indent
We can now give the intrinsic description of linear CR quaternionic structures.

\begin{prop} \label{prop:cr_q_intrinsic}
There exists a natural correspondence between the following:\\
\indent
{\rm (i)} Linear CR quaternionic structures on $U$;\\
\indent
{\rm (ii)} Quaternionic vector subspaces $C\subseteq U^{\Hq}$ such that\\
\indent
\qquad{\rm (ii1)} $C\cap\bigcap_{\t\in\Lambda}\t(C)=0$\,,\\
\indent
\qquad{\rm (ii2)} $C+\s(C)=U^{\Hq\!}$, for any $\s\in\Lambda$\,.
\end{prop}
\begin{proof}
If $(E,\iota)$ is a linear CR quaternionic structure on $U$ then $C=\bigl(\iota^{\Hq\!}\bigr)^{-1}(C_E)$ satisfies
assertion (ii)\,, where $C_E$ is the quaternionic vector subspace of $E^{\Hq}$ given by assertion (iii) of Proposition \ref{prop:q_Bonan}\,.\\
\indent
Conversely, if $C$ is as in (ii) then on defining $E=U^{\Hq\!}/C$ and $\iota$ to be the composition of the inclusion
of $U$ into $U^{\Hq\!}$ followed by the projection from the latter onto $E$ we obtain the corresponding linear CR quaternionic structure.
\end{proof}

\indent
Finally, by duality, we also have.

\begin{prop} \label{prop:co-cr_q_intrinsic}
There exists a natural correspondence between the following:\\
\indent
{\rm (i)} Linear co-CR quaternionic structures on $U$;\\
\indent
{\rm (ii)} Quaternionic vector subspaces $B\subseteq U^{\Hq}$ such that\\
\indent
\qquad{\rm (ii1)} $U^{\Hq\!}=B+\sum_{\t\in\Lambda}\t(B)$\,,\\
\indent
\qquad{\rm (ii2)} $B\cap\s(B)=0$\,, for any $\s\in\Lambda$\,.
\end{prop}

\end{document}